\documentclass{amsart}
\usepackage{tikz}
\usetikzlibrary{snakes}
\usetikzlibrary{patterns}
\theoremstyle{plain}
\newtheorem{theorem}{Theorem}[section]
\newtheorem{lemma}[theorem]{Lemma}
\newtheorem{corollary}[theorem]{Corollary}
\newtheorem{proposition}[theorem]{Proposition}
\newtheorem*{corollary*}{Corollary}

\DeclareMathOperator{\var}{Var}
\theoremstyle{definition}
\newtheorem{definition}[theorem]{Definition}
\newtheorem{example}[theorem]{Example}
\def\R2{\mathbb{R}^2}
\newcommand{\onetorus}{\mathbb{S}^1}

\title[Ergodic group extensions]{Ergodic infinite group extensions of geodesic flows on translation surfaces}

\author{David Ralston}
\address{Ben Gurion University, Department of Mathematics\\ POB 653 \\ Beer Sheva, 84105\\ ISRAEL}
\email{ralston.david.s@gmail.com}

\author{Serge Troubetzkoy}
\address{Aix-Marseille University, CPT, IML, Frumam, Marseille}
\email{troubetz@iml.univ-mrs.fr}

\thanks{The first author is supported by the Center for Advanced Studies at Ben Gurion University of the Negev as well as the Israel Council for Higher Education, and was partially supported by the Erwin Schr\"{o}dinger International Institute for Mathematical Physics during preparation of this manuscript. This research is partially supported by the ANR project Perturbations.}
\date{\today}

\begin{document}
\begin{abstract}
We show that generic infinite group extensions of geodesic flows on square tiled translation surfaces are ergodic in almost every direction, subject to certain natural constraints.  Recently K. Fr\c{a}czek and C. Ulcigrai have shown that certain concrete staircases, covers of square-tiled surfaces, are not ergodic in almost every direction. In contrast
we show the almost sure ergodicity of other concrete staircases.  An appendix provides a combinatorial approach for the study of square-tiled surfaces.
\end{abstract}
\maketitle

\section{Introduction} In 1986 Kerchoff, Masur and Smillie showed that the
the geodesic flow on a compact translation surface is
ergodic in almost every direction \cite{1986}.  On the other hand,  the study  of  the dynamical properties on periodic infinite translation
surfaces is in its infancy.  There is a natural necessary and sufficient condition for the recurrence of the geodesic flow
in almost every direction for $\mathbb{Z}$ or $\mathbb{R}$ covers \cite{HoW}.  Generic non-periodic translation surfaces
are almost surely recurrent \cite{T}.  The almost sure diffusion rate in the periodic full occupancy 
wind-tree model, a certain important example of infinite translation
surface, has also recently been understood \cite{DHL}.  We turn to the question of ergodicity.  In \cite{hooper-hubert-weiss}, a natural example the ``simple staircase'' was shown to be ergodic in a.e.\ direction.  In fact, the authors noticed that the ergodicity reduces to a classically studied question of ergodicity of cylinder flows (group valued skew products) over circle rotations initiated by Conze and studied by many authors, see \cite{conze-cocycles-rotation} for a good survey.  The question of ergodicity for other periodic translation surfaces reduces to the question of ergodicity of cylinder flows over interval exchange transformations.

Recently, Fraczek and Ulcigrai have shown that certain periodic infinite translation surfaces are not ergodic in almost every direction
\cite{fraczek-ulcigrai}.  More specifically, if the compact square-tiled translation surface $M$ is in the stratum
$\mathcal{H}(2)$, and the infinite translation surface $\tilde{M}$ is a $\mathbb{Z}$ unramified cover of $M$, then the infinite surface is not ergodic in a.e.\ direction.
They showed a similar result for some $\mathbb{Z}^2$ extensions such as the periodic full occupancy wind-tree model.

In this article, we study  the case when the surface $\tilde{M}$ is a ramified cover of $M$. This is generically the case.
In the language of ergodic theory the flow on $\tilde{M}$ is a $G$-valued skew product over the geodesic flow on $M$
for some Abelian locally compact group $G$.  Our result is a sufficient condition for both recurrence and almost sure ergodicity
to hold.  This condition is based on ideas developed by M. Boshernitzan \cite{boshernitzan} and K. Schmidt \cite{schmidt}.

While our theorem is somewhat technical to state, it is applicable to many concrete examples.  Suppose that $M$ is
a square tiled surface with a single-cylinder direction (for example the torus),
then any reasonable method of producing a $\mathbb{Z}^d$
(or more generally an Abelian locally compact group $G$) infinite surface $\tilde{M}$ which is a skew product over $M$ is generically
ergodic in almost every direction.  
We also  produce concrete examples of staircases ergodic in almost every direction such as the staircase in Figure \ref{Some staircases}.
In the appendix, we give a procedure for determining which square-tiled surfaces have a single-cylinder direction, as well as a quick method for constructing such surfaces and determining their genus as well as the number and type of singularities.

\makeatletter
\def\Ddots{\mathinner{\mkern1mu\raise\p@
\vbox{\kern7\p@\hbox{.}}\mkern2mu
\raise4\p@\hbox{.}\mkern2mu\raise7\p@\hbox{.}\mkern1mu}}
\makeatother

\begin{figure}[htb]
\center{\begin{tikzpicture}
\draw (6,0) rectangle (7,1);
\draw (7,0) rectangle (8,1);
\draw (8,1) rectangle (9,2);
\draw (7,1) rectangle (8,2);
\draw (9,1) rectangle (10,2);
\draw (9,2) rectangle (10,3);
\draw (10,2) rectangle (11,3);
\draw (10,3) rectangle (11,4);
\draw (11,3) rectangle (12,4);
\draw (12,3) rectangle (13,4);
\node at (13,4.5){$\Ddots$};
\node at (6,-0.5){$\Ddots$};
\end{tikzpicture}}
\caption{\label{Some staircases}The periodic staircase with alternating steps of length 2 and 3 is ergodic in almost every direction.
Opposite sides are identified.}
\end{figure}
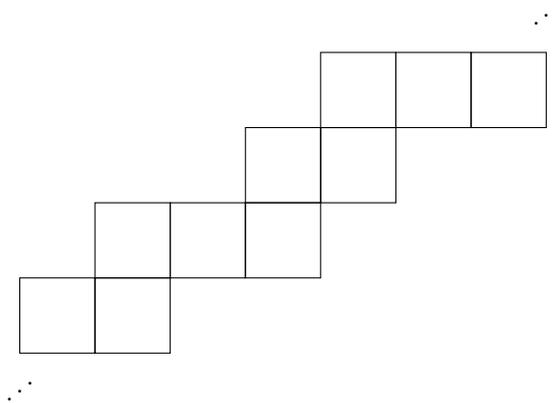

J. Chaika and P. Hubert have recently announced the following related result \cite{chaika-hubert}:
for a given $f:[0,1) \to \mathbb{R}$ with integral 0,
which is the finite sum of characteristic functions of intervals,
for a set of full measure of interval exchange transformations $T$ the
skew  product $T_f(x,i)=(Tx,i+f(x))$  into the closed subgroup of $\mathbb{R}$ generated by the values of $f(x)$ is ergodic.  Their methods
of proof are close to ours.

\section{Translation surfaces}\label{section - surfaces background}

In this section we will give background material on translation surfaces as well as some definitions for later use, see \cite{MR2186246, MR2186247, MR2261104} for details.  A translation surface is a $2$-dimensional manifold $M$ and a finite set of points $D = \{d_1,d_2,\dots,d_m\}$ and an open cover of $M \setminus D$ by sets
$\{U_{\alpha}\}$ together with charts $\phi_{\alpha} : U_{\alpha} \to \R2$ such that for all $\alpha$, $\beta$, with $U_{\alpha} \cap U_{\beta} \ne \emptyset$, $\phi_{\alpha} \circ \phi_{\beta}^{-1}$ is a translation on its domain of definition and at each singular point $d_i \in D$ the surface has a $2 \pi n$ cone singularity for some  integer $n \ge 1$.

The surface $M$ has a flat metric obtained by pulling back the Euclidean metric on the plane via the coordinate charts.  This metric is defined away from the set $D$.  In this metric geodesics that do not go through singularities project via the charts to straight lines in the plane in a fixed direction, and such geodesics are either periodic or simple (do not intersect themselves).  If the geodesic in direction $\theta$ starting at
$x \in M$ avoids $D$ then we define the geodesic flow $\varphi_t^\theta(x)$ to the the point obtained after moving in the direction $\theta$ for a time $t$ starting at $x$.  Since $\theta$ is usually fixed,  we will often suppress the $\theta$ dependence of the flow and simply write $\varphi_t(x)$.
If the geodesic hits $D$, then the geodesic flow $\varphi_t(x)$ is defined up to this time. The geodesic flow preserves the surface area measure 
$\mu$.

Let $\Gamma \subset SL_2(\mathbb{R})$ denote the \textit{Veech group} of the surface, the group of affine transformations (seen in $SL_2(\mathbb{R})$ as acting on the charts of $M$) which preserve the structure of $M$, and therefore permute the singularities $D$.  If $SL_2(\mathbb{R})/\Gamma$ is of finite volume (i.e.\ $\Gamma$ is a \textit{lattice}), we say that $M$ is a \textit{Veech surface}.

Throughout the article we will assume that $M$ is a compact translation surface; by \cite{1986} the geodesic flow is (uniquely) ergodic in almost every $\theta$.  We will form \textit{skew products} over these flows in the following manner: let $G$ be a locally compact Abelian group with translation-invariant metric $\| \cdot \|$, whose identity element we denote $0$, with Haar measure $\nu$, and for a finite collection of geodesic path segments $\{\gamma_i\}$ (disjoint without loss of generality) whose union we denote $\gamma$, let $f:\gamma \rightarrow G$ be constant on each $\gamma_i$ and $0$ elsewhere.  Given a point $x \in M$ and a direction $\theta$ which is not parallel to any $\gamma_i$, denote
\[ S_t(x) = \sum_{\substack{ 0 \leq s <t \\ f(\varphi_s(x)) \neq 0}}f(\varphi_s(x)),\]
\[S_t^{-1}(g) = \left\{x \in M : S_t(x) = g\right\},\quad B_{\epsilon}(g) = \left\{ y \in G : \|y - g\|<\epsilon\right\}.\]

Then the transformation $\tilde{\varphi}_t$, from $\tilde{M}=M \times G$ to itself, is given by
\[\tilde{\varphi}_t(x,g) = (\varphi_t(x),g+S_t(x)),\] and it is this flow which we study; $\tilde{\varphi}_t$ preserves the measure $\tilde{\mu}=\mu \times \nu$.  An \textit{essential value} of a general skew product $\{\tilde{M}, \tilde{\mu}, \tilde{\varphi}_t\}$ is some $g \in G$ such that for every $\epsilon>0$ and $A \subset M$ with $\mu(A)>0$, there is some $t>0$ such that
\[ \mu\left(A \cap \varphi_t(A) \cap S_t^{-1}(B_{\epsilon}(g)) \right) >0.\]
We denote by $E(\tilde{\varphi}_t)$ the union of all essential values.  Then if $E(\tilde{\varphi}_t)$ is nonempty, it is a closed subgroup of $G$, and $\{\tilde{M},\tilde{\varphi}_t\}$ is ergodic if and only if $E(\tilde{\varphi}_t) = G$ and $\{M, \varphi_t\}$ is ergodic (see \cite{A} for example).  If $0 \in E(\tilde{\varphi}_t)$, then the system is called \textit{recurrent}.

\begin{definition}\label{definition - quasirigidity sets} We will call a sequence of sets $C_n$ a \textit{quasi-rigidity sequence of sets} in the direction 
$\theta$ if there is a sequence $t_n \rightarrow \infty$, a sequence $\epsilon_n \rightarrow 0$ and a fixed $\epsilon>0$ such that
\begin{itemize}
\item $\mu(C_n) \geq \epsilon,$
\item $\forall x \in C_n$, $\|\varphi_{t_n}(x) - x\| \leq \epsilon_n$, and
\item $\mu(C_n \triangle \varphi_{1}C_n) \leq \epsilon_n$.
\end{itemize}
That is, each set is of nontrivial measure, the map $\varphi_{t_n}$ acts almost like the identity on these sets, and they are nearly invariant under the flow $\varphi_1$ (the use of $\varphi_1$ is arbitrary; any fixed positive time would do).
\end{definition}

\begin{lemma}\label{lemma - quasiperiodic citation}
Suppose that $\{M, \varphi_t\}$ is ergodic and that for some $g \in G$ we have a quasi-rigidity sequence of sets $C_n$ such that
\[ \lim_{n \rightarrow \infty}\sup_{x \in C_n} \left\| S_{t_n}(x) - g\right\| = 0.\]  Then $g \in E(\tilde{\varphi}_t)$.
\begin{proof}
This is merely a restatement for continuous flows of \cite[Cor. 2.6]{conze-fraczek}.  In that reference, it is assumed that $G=\mathbb{R}^d$, but the techniques apply to any locally compact Abelian group with Haar measure and translation-invariant metric, as we have assumed.  A similar remark appears as \cite[Thm. 11]{hubert-weiss}; the notion of restricting a search for essential values to certain subsets whose measure is bounded away from zero is a common step in the study of periodic group extensions of finite-measure systems.
\end{proof}
\end{lemma}

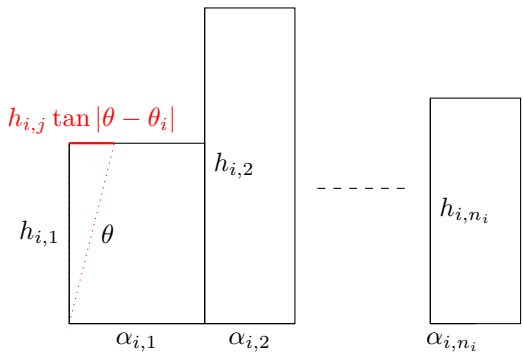
\begin{figure}[bh]
\center{
\begin{tikzpicture}[scale=.6]
\draw (0,0) rectangle (3,4);
\draw (0,0) -- node[below]{$\alpha_{i,1}$} (3,0);
\draw (0,0) -- node[left]{$h_{i,1}$} (0,4);
\draw[dotted] (0,0) -- (0,4);
\node[right] at (0.5,2){$\theta$};
\draw[thick, red] (0,4) -- node[above]{$h_{i,j} \tan |\theta - \theta_i|$} (1,4);
\draw[dotted, red] (0,0)--(1,4);

\draw (3,0) rectangle (5,7);
\draw (3,0) -- node[below]{$\alpha_{i,2}$} (5,0);
\draw (3,0) -- node[right]{$h_{i,2}$}(3,7);
\draw[dashed] (5.5,3) -- (7.5,3);
\draw (8,0) rectangle (10,5);
\draw (8,0) -- node[below]{$\alpha_{i,n_i}$} (9,0);
\draw (8,0) -- node[right]{$h_{i,n_i}$} (8,5);
\end{tikzpicture}}
\caption{\label{figure - cylinders}Decomposition of $\{M, \varphi_t\}$ into periodic cylinders in direction $\theta_i$ of widths $\alpha_{i,j}$ and heights $h_{i,j}$; the top and bottom of each rectangle are identified, and $\varphi_t$ flows vertically at unit speed.  The error compared to the flow in direction $\theta$ relative to the width of the cylinders is bounded for periodic approximations, bounded away from zero for fairly good periodic approximations, and converging to zero for good periodic approximations.}
\end{figure}

\begin{definition}\label{definition - periodic cylinders}
If the geodesic flow $\varphi_t$ in a particular direction $\theta$ decomposes the space $M$ into a finite union of sets on which $\varphi_t$ is periodic (see Figure \ref{figure - cylinders}), the direction $\theta$ is said to \textit{admit a representation by periodic cylinders}.  Given a sequence $\{\theta_i\}$ of directions which admit a representation by periodic cylinders, denote the cylinders by $A_{i,j}$, the width of $A_{i,j}$ by $\alpha_{i,j}$, and the height of each cylinder (i.e.\ the period of the geodesic flow in this cylinder) by $h_{i,j}$.

The quantity
\[E_i= \max\left\{\frac{h_{i,j} \tan|\theta - \theta_i|}{\alpha_{i,j}}: j=1,2,\ldots,n_i \right\}\]
will be called the \textit{error of the approximation $\theta_i$}.  If $ \theta_i \rightarrow \theta$ and $E_i \in O(1)$, we  say that $\{\theta_i\}$ is a \textit{sequence of periodic approximations} to the flow in the direction $\theta$.  Note that if $\theta_i \rightarrow \theta$ and $\theta$ does \textit{not} admit a representation by periodic cylinders (and only countably many directions, those of \textit{saddle connections}, i.e.\ orbit segments which begin and end at singular points, may have such a representation), then necessarily we have
\[\lim_{i \rightarrow \infty} \min_j h_{i,j} = \infty.\]
If furthermore $E_i \rightarrow 0$ we  call the sequence $\{\theta_i\}$ a \textit{good sequence of periodic approximations}.  If there is some $\delta>0$ such that for all $i$ we have $E_i \in (\delta,1/2)$ we will call the sequence a \textit{fairly good sequence of periodic approximations}.  See again Figure \ref{figure - cylinders}.  If $\inf \left\{\mu(A_{i,j})= h_{i,j} \cdot \alpha_{i,j} : i =1,2,\ldots, \, j=1,2,\ldots,n_i \right\} > 0$, we will say that this sequence \textit{has cylinders of comparable measures}.
\end{definition}

\begin{lemma}\label{lemma - our quasirigidity sets}
Suppose that $\theta$ admits a sequence of good periodic approximations $\{\theta_i\}$ with cylinders of comparable measures, and let $\{i_k\}$ be an increasing sequence of positive integers.  For some fixed $\delta>0$, consider a sequence of \textit{subcylinders} $C_k$; connected subsets of the cylinder $A_{i_k,j}$ (for some $j$) which are invariant under the flow in direction $\theta_{i_k}$, and whose measure is at least $\delta \mu(A_{i_k,j})$.  Assume further that the boundary of each $C_k$ is separated from the boundary of $A_{i_k,j}$ by a distance of at least $\delta \alpha_{i_k,j}$.  Then the $\{C_k\}$ form a quasi-rigidity sequence of sets.
\begin{proof}
That $\mu(C_k)$ is bounded away from zero follows from the condition that the cylinders have comparable measures.  As the sequence $\theta_i$ was taken to be a sequence of \textit{good} periodic approximations, the relative error converges to zero.  It now follows from the fact that each $C_k$ is bounded away from the boundary of $A_{i_k,j}$ that for each $x \in C_k$, the orbit through time $h_{i_k,j}$ of $x$ remains in $A_{i_k,j}$, from which it follows that
\[ \| \varphi_{h_{i_k,j}}(x)-x\| \leq E_{i_k} \alpha_{i_k,j}\xrightarrow{k \rightarrow \infty} 0.\]
Finally, we have
\[ \mu(\varphi_1 C_k \triangle C_k) \leq 2 \cdot E_{i_k} \alpha_{i_k,j}\] (the difference is exactly those $x$ so close to the boundary of $C_k$ that they flow out of the subcylinder in unit time), which certainly converges to zero.
\end{proof}
\end{lemma}

\begin{definition}
Suppose that the geodesic flow in direction $\theta$ admits a representation by a single periodic cylinder.  Then $\theta$ will be called a \textit{single-cylinder direction}.
\end{definition}

Single-cylinder directions are of great convenience in studying translation surfaces.  It was shown in \cite{hubert-lelievre} that if $M$ is square-tiled with a prime number of tiles and in $\mathcal{H}(2)$ (genus two with a single singularity), then $M$ has a single-cylinder direction.  In \cite[Cor. A.2]{mcmullen} McMullen extended this result  to all square-tiled surfaces in $\mathcal{H}(2)$.  
Kontsevich and Zorich showed that  there is a dense set of square tiled surfaces in any stratum with single cylinder directions \cite[Lemma 18]{KZ}.
Finally Lanneau and Nguyen have identified certain substrata of $\mathcal{H}(4)$ and $\mathcal{H}(6)$ which have single cylinder directions \cite{L}.
See Appendix \ref{section - make one cylinder directions} for a general technique for constructing square-tiled surfaces with a single-cylinder direction.

The following is a standard argument in the study of translation surfaces, but hardly elementary:
\begin{lemma}\label{lemma - approximate full measure of directions}
Suppose that $\{M, \varphi_t\}$ has at least one direction $\theta'$ which admits a representation by $N$ periodic cylinders, and suppose that $M$ is a Veech surface with Veech group $\Gamma$.  Then for almost every direction $\theta$, there is a sequence $\{\theta_i\}$ of directions, chosen from the directions $\Gamma (\theta')$, which are a good sequence of periodic approximations, with $N$ cylinders of comparable measures, to the flow in the direction $\theta$.  Furthermore, for almost every direction $\theta$ there is a sequence $\{\theta_i\}$ taken from the directions $\Gamma (\theta')$ of fairly good periodic approximations with $N$ cylinders of comparable measures.
\begin{proof}
Any element $\gamma \in \Gamma$ will transform the flow in direction $\theta'$ to a flow in another direction which must necessarily also admit a representation by periodic cylinders.  The number of cylinders and their measures are not changed by the action of $\gamma \in SL(2,\mathbb{R})$, so all that must be shown is that for almost every $\theta$, a sequence may be chosen out of this orbit which are a \textit{good} or \textit{fairly good} sequence of periodic approximations to the flow in the direction of $\theta$.  

The assumption that $M$ is a Veech surface ensures that $\Gamma$, the Veech group, is \textit{Fuchsian of the first kind}.  The fact that almost-every direction has a sequence of good approximations coming from the orbit under $\Gamma$ of the single-cylinder direction follows directly from \cite[\S9]{Patterson} (see \cite{MR688349} for a more general case).  We state the application of this theorem  to our setting in the language of
this article: for every  $\varepsilon > 0$, there exists a set
of full measure $\Theta(\varepsilon) \subset \mathbb{S}^1$ such that for any $\theta \in \Theta(\varepsilon)$, there exists a sequence of $N$
cylinder directions with slopes $p_n/q_n$ such that $|\theta - p_n/q_n| \le \varepsilon/q_n^2$. Similarly as in \cite{HLT},  the conclusions for good
approximations follow from Patterson's theorem.  The appropriate lower bound for fairly good approximations  follows from the ergodicity techniques used to prove similar statements in \cite[\S 2.2]{hubert-weiss}: the orbit of a completely periodic direction under $\Gamma$ will intersect arbitrary sets of positive measure infinitely often in $SL_2(\mathbb{R})/\Gamma$.  This quotient is of finite volume and has finitely many cusps by assumption, so a nontrivial annular neighborhood around a generic slope will also contain infinitely many points from the orbit $\Gamma (\theta'$); the use of a fixed annulus as opposed to a sequence of shrinking neighborhoods distinguishes fairly good approximations from good approximations.
\end{proof}
\end{lemma}
In particular, if $M$ is a Veech surface, then if the flow in direction $\theta'$ is not periodic the surface admits a representation by periodic cylinders of comparable measures with whatever strength in approximation we desire.  The conclusion of Lemma \ref{lemma - approximate full measure of directions} can in fact be strengthened to state that the set of exceptional directions (those which cannot be approximated to the desired accuracy from the orbit under $\Gamma$ of any given single-cylinder direction) is not just of measure zero, but indeed of Hausdorff dimension less than one.  See again \cite{hubert-weiss} for similar statements; we state only that almost-every direction can be so approximated in the interest of simplicity and generic appeal.

\section{Generic Constructions}\label{section - generic}

\begin{definition}\label{definition - self avoiding}
Let $\theta$, the direction of the geodesic flow, be fixed such that $\{M, \varphi_t\}$ admits a sequence of periodic approximations $\{\theta_i\}$.  Let $\tilde{D}$ denote an arbitrary finite subset of $M$.  Then a point $x \in (\tilde{D} \setminus D)$ is said to \textit{avoid $D \cup \tilde{D}$ in direction $\theta$} if for some fixed $\epsilon>0$ there is a sequence $\{\epsilon_i\}$ so that each of the following holds:
\begin{itemize}
\item for each $i$, $x$ belongs to the interior of some $A_{i,j}$,
\item the $\epsilon_i$-neighborhood around the $h_{i,j}\sec |\theta_i - \theta |$ orbit of $x$ \textit{in the direction of $\theta$} remains in $A_{i,j}$ (restricted to those $j$ for which $x \in A_{i,j}$) and does not intersect $D \cup (\tilde{D} \setminus x)$, and 
\item $\inf_{i \rightarrow \infty} \epsilon_i h_{i,j} \geq \epsilon$.
\end{itemize}
A sequence of neighborhoods $C_i$ for which $x$ satisfies this definition for the same $\epsilon>0$ will be called a \textit{sequence of orbit-neighborhoods for $x$}.  See Figure \ref{figure - subcylinders} for an illustration (the sets $C_{i,1}$ and $C_{i,2}$ will be defined later).  If every point in $\tilde{D}$ avoids $D \cup \tilde{D}$, and if $D \cap \tilde{D} = \emptyset$, then we will say that $\tilde{D}$ is a \textit{self-avoiding set in direction $\theta$}.
\end{definition}

Note that we merely need \textit{some} sequence of periodic approximations for each $\theta$; in situations where a canonical method for finding approximations exists, a subsequence of this canonical set suffices.

\begin{proposition}\label{prop - avoiding points generic in measure}
Suppose $\theta$ has a good sequence of periodic approximations and fix $P \in \mathbb{N}$.  Then almost every element $\tilde{D} \in M^P$ (with respect to the product measure) is a self-avoiding set in the direction $\theta$.
\begin{proof}
It suffices to show that the proposition is true for $P=1$; larger sets may then be inductively constructed one point at a time by presuming $x_1$ through $x_{P-1}$ are in $D$ for the purposes of finding $x_P$.  Let $\theta$ be fixed such that $\{\theta_i\}$ is a good sequence of periodic approximations.  For a fixed $\epsilon>0$, let $B_i = B_i(\epsilon)$ be the complement of the $\epsilon'$-neighborhood of the boundaries of the cylinders, where
\[ \epsilon' = \frac{\epsilon \min_j \{\alpha_{i,j}\}}{2},\] so that $\mu(B_i) \geq (1-\epsilon)$.  See Figure \ref{figure - shaved cylinders}.  Now note that if $x$ belongs to infinitely many $B_i(\epsilon)$, it is self-avoiding singleton: for sufficiently large $i$ we have the error $E_i<\epsilon$.  As
\[ \mu \left( \limsup_{i \rightarrow \infty} B_i \right) \geq \limsup_{i \rightarrow \infty} \mu(B_i),\] for this choice of $\epsilon$ the set of self-avoiding singletons in direction $\theta$ is of measure at least $(1-\epsilon)$, and $\epsilon$ was arbitrary.
\end{proof}
\end{proposition}
\begin{figure}[b]
\center{
\begin{tikzpicture}[yscale=.75]
\draw[pink, fill=pink] (0,0) rectangle (1/8,4);
\draw[pink, fill=pink] (3-1/8,0) rectangle (3,4);
\draw (0,0) rectangle (3,4);
\draw (0,0) -- node[below]{$\alpha_{i,1}$} (3,0);
\draw (0,0) -- node[right]{$h_{i,1}$} (0,4);
\draw[pink, fill=pink] (3,0) rectangle (3+1/8,7);
\draw[pink, fill=pink] (5-1/8,0) rectangle (5,7);
\draw (3,0) rectangle (5,7);
\draw (3,0) -- node[below]{$\alpha_{i,2}$} (5,0);
\draw (3,0) -- node[right]{$h_{i,2}$}(3,7);
\draw[dashed] (5.5,3) -- (7.5,3);
\draw[pink, fill=pink] (8,0) rectangle (8+1/8,5);
\draw[pink, fill=pink] (9-1/8,0) rectangle (9,5);
\draw (8,0) rectangle (9,5);
\draw (8,0) -- node[below]{$\alpha_{i,n_i}$} (9,0);
\draw (8,0) -- node[right]{$h_{i,n_i}$} (8,5);
\end{tikzpicture}
}
\caption{\label{figure - shaved cylinders}The $\epsilon$-proportion removed in constructing $B_i(\epsilon)$.}
\end{figure}
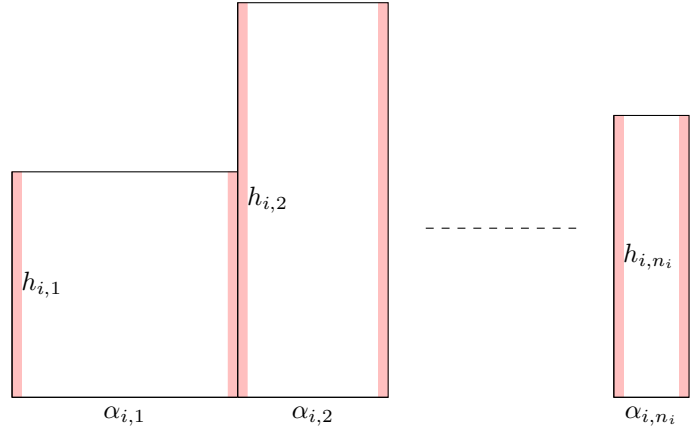

\begin{corollary}\label{Corollary - fubini trick}
Suppose that almost every $\theta$ admits a sequence of good approximations.  Then almost every subset $\tilde{D} \in M^P$ is self-avoiding in almost every direction.
\begin{proof}
We certainly have almost surely that $\tilde{D} \cap D = \emptyset$.  For a full-measure set of $\theta$, a full-measure set of $\tilde{D}$ is self-avoiding in direction $\theta$.  Apply Fubini's theorem to reverse the order of quantifiers.
\end{proof}
\end{corollary}

Let $S=\{s_1,\ldots,s_n\}$ be a finite subset of $M$ (possibly including points in $D$), $L_i > 0$  and $\tau_i \in \onetorus$ ($i=1,\dots,n$) be for the
moment arbitrary.  Let $\gamma_r$ be a geodesic of length $L_r>0$ in direction $\tau_r$ originating at $s_r$. Note that we write ``a'' geodesic since if $s_r \in D$ there may be several possible such geodesics and if a geodesic originating at $s_r$ arrives at a point of $D$, there are multiple possible continuations.  Although generically this will not happen, a priori we allow all such geodesic segments. Let $a_{r,1}, \ldots, a_{r,N(r)}$ be a finite set of points in $\gamma_r$ ($N(r)$ is arbitrary but finite); assume that $a_{r,N(r)}=\varphi_{L_r}(s_r)$ (the terminus of $\gamma_r$) and set $a_{r,0}=s_r$.  Denote by $A$ the (finite) collection of all $a_{r,j}$ together with any $d \in D$ such that $d \in \gamma_i$ for some $i$.

\begin{corollary}
For fixed $n$, for almost every construction of $\gamma_r$ and $A$ (with initial points $s_i$ chosen according to a probability measure on $M$ absolutely continuous with respect to $\mu$, lengths $l_r$ chosen according to some probability measure absolutely continuous with respect to Lebesgue on $\mathbb{R}^+$ and directions $\tau_r$ by a measure absolutely continuous with respect to Lebesgue measure on $\onetorus$, and $A$ then chosen according to a measure on $\gamma$ which is absolutely continuous with respect to Lebesgue measure), the set $A$ is a self-avoiding set.
\begin{proof}
We apply Proposition \ref{prop - avoiding points generic in measure}, noting simply that we certainly have almost-surely that no $\gamma_i$ crosses a singularity $d \in D$.
\end{proof}
\end{corollary}

Let $G$ be an Abelian locally compact group with translation invariant metric endowed with Haar measure $\nu$, and we will use the segments $\gamma_i$ to define our skew product $\{\tilde{M},\tilde{\varphi}_t\}$ as in \S \ref{section - surfaces background}.  Define a function $f: M \rightarrow G$ to be piecewise constant on each $\gamma_i$ with discontinuities a subset of $A$; denote the value of $f$ between $a_{i,j-1}$ and $a_{i,j}$ by $f_{i,j}$.  Finally, denote
\[ \sigma_{i,j} = \begin{cases} f_{i,1} & (j=0)\\ f_{i,j+1}-f_{i,j} & (j=1,2, \ldots N(i)-1)\\ -f_{i,N(i)} & (j=N(i)). \end{cases}\]
The function $f$ is defined to be identity off the segments $\gamma_i$; the values $\sigma_{i,j}$ are the `jumps' seen in $f$ while flowing along the path $\gamma_i$.

\begin{definition}\label{definition - essential point}
Let $\theta$ be fixed.  A point $x \in (\tilde{D} \setminus D)$ which avoids $(\tilde{D} \cup D)$ with orbit neighborhoods $C_{k}$ of heights $h_k \rightarrow \infty$ will be called \textit{essential in direction $\theta$} if there is some $P$ such that along an infinite subsequence of orbit neighborhoods we have
\[ \sup_{k=1,2,\ldots}\sup_{y \in C_k} \left| S_{h_k}(y) \right| \leq P.\]
\end{definition}

Showing that points are essential is in general a very demanding requirement.  Single-cylinder directions, however, provide a ready approach.

\begin{proposition}[Koksma-type inequality]\label{proposition - single cylinder avoiding points are essential}
Suppose that $M$ is a translation surface and the function $f$ is defined as above.  Assume further that $\theta'$ is a single cylinder direction, with the cylinder height given by $h$ (and width $1/h$).
Suppose that the flow in the single cylinder direction $\theta'$,  sees an average value of the identity in $G$.  
Then, for every $x$ not on the boundary of the cylinder we have
\[ \| S_{h}(x)\| \leq 4\sum_{j=1}^{N(i)}\sum_{i=1}^n \|f_{i,j}\|,\] where the sum $S_h(x)$ is with respect to the flow in the direction $\theta'$.
\begin{proof}
We denote by $q_{i,j}$ the projected length of the geodesic segment connecting $a_{i,j-1}$ to $a_{i,j}$ \textit{in the direction transverse to $\theta'$}.
That the average value of the sums in the cylinder is identity is
equivalent to
\begin{equation}\label{eqn - zero integral condition}\sum_{j=1}^{N(i)}\sum_{i=1}^n f_{i,j}q_{i,j} =0.\end{equation}
This condition is sufficient for recurrence of the flow in the event 
that $G = \mathbb{Z}$ or $\mathbb{R}$ (see \cite{schmidt2}), and is clearly necessary in all cases where $G$ contains elements of infinite order.

Given a particular segment, the points $a_{i,j-1}$ and $a_{i,j}$ are in the cylinder (possibly on the boundary), and the segment between them completely crosses the cylinder some number of times.  Denote this number by $P_{i,j}$.  There are also possibly two partial crossings corresponding to the segment leaving $a_{i,j-1}$ and arriving into $a_{i,j}$ (in $P_{i,j} = 0$ then there is only one such segment); denote by $p'_{i,j}$ and $p''_{i,j}$ respectively the lengths of these two segments when measured in the transverse direction to the flow.  We have
\[ 0 \leq p'_{i,j},p''_{i,j} < \frac{1}{h}, \quad 0 \leq q_{i,j}- \frac{P_{i,j}}{h} \leq \frac{2}{h},\]
the latter of which we rewrite as
\[ |P_{i,j}-q_{i,j}h| \leq 2.\]

As the entire space is shown in this single periodic cylinder, the entire segment is represented by these pieces, so \eqref{eqn - zero integral condition} translates to
\begin{equation}\label{eqn - zero integral translated} \sum_{j=1}^{N(i)} \sum_{i=1}^n f_{i,j}\left(\frac{P_{i,j}}{h} +p'_{i,j} + p''_{i,j}\right)=0.\end{equation}

The periodic orbit of any point $x$ in the direction $\theta'$ (ignoring the orbits containing the points $a_{i,j}$) crosses the segment from $a_{i,j-1}$ to $a_{i,j}$ a total of $P_{i,j}+\phi_{i,j}(x)$ times, where
\[ \phi_{i,j}(x) \in \left\{0,1,2 \right\},\] according to whether it crosses neither `partial segment,' one of them, or both.
Then we have
\begin{align*}
S_{h}(x) & = \sum_{j=1}^{N(i)}\sum_{i=1}^n f_{i,j}(P_{i,j}+\phi_{i,j}(x))\\
\|S_{h}(x)\| &\leq \left\| \sum_{j=1}^{N(i)}\sum_{i=1}^n f_{i,j}P_{i,j} \right\| + 2\sum_{j=1}^{N(i)}\sum_{i=1}^n \|f_{i,j}\|\\
&\leq h\left\| \sum_{j=1}^{N(i)}\sum_{i=1}^n f_{i,j}q_{i,j} \right\| + 4\sum_{j=1}^{N(i)}\sum_{i=1}^n \|f_{i,j}\|\\
&=4\sum_{j=1}^{N(i)}\sum_{i=1}^n \|f_{i,j}\|. \qedhere
\end{align*}
\end{proof}
\end{proposition}

\begin{theorem}\label{theorem - essential points give finite essential values}
Let $a=a_{j,i}$ be an essential point with respect to a good sequence of periodic approximations.  Then $\sigma=\sigma_{i,j}$ is an essential value of the skew product $\{\tilde{M},\tilde{\varphi}_t\}$ for the flow in direction $\theta$.  In particular, the skew product is recurrent.
\begin{proof}
Using a good sequence of periodic approximations allows us to consider for almost every $x$, for sufficiently large $i$ the flow in direction $\theta$ and the flow in direction $\theta_i$ are indistinguishable through length $h_{i,j}$, at least insofar as the sums $S_{h_{i,j}}(x)$ are concerned.  The point $a$ has a sequence of orbit-neighborhoods $C_k$ which remain in the same periodic cylinder $A_{i_k,j}$.  The only discontinuity of $f$ in $C_k$ is the point $a$ by definition. So we split the orbit-neighborhood into two sets of equal measure $C_{k,1}$ and $C_{k,2}$ so that the sum $S_{h_{i_k}}(x)$ is constant on each.  Furthermore, the difference between the two values taken is exactly $\sigma$.  See Figure \ref{figure - subcylinders}.  Without loss of generality, then, assume that for $x \in C_{k,1}$ and $y \in C_{k,2}$ we have
\[ S_{h_{i_k,j}}(y) = S_{h_{i_k,j}}(x)+\sigma.\]
As there is a uniform bound $P$ across all values and $G$ is locally compact, by passing to a subsequence (which we continue to denote $\{i_k\}$) we may find some limit
\[ \lim_{k \rightarrow \infty} S_{h_{i_{k}}}(y \in C_{k,1}) = g, \quad \lim_{k \rightarrow \infty} S_{h_{i_{k}}}(y \in C_{k,1}) = g+\sigma.\]

Both $C_{k,1}$ and $C_{k,2}$ are a quasi-rigidity sequence sets by Lemma \ref{lemma - our quasirigidity sets}, so by Lemma \ref{lemma - quasiperiodic citation}, both $g$ and $g+\sigma$ are in $E(\tilde{\varphi}_t)$.  As the set of essential values is a group, $\sigma \in E(\tilde{\varphi}_t)$.
\end{proof}
\end{theorem}

\begin{figure}[hb]
\center{
\begin{tikzpicture}[xscale=2]
\draw[fill=cyan, xslant=1/10] (1,0) rectangle (2-1/3,5);
\draw[fill=lime, xslant=1/10] (2-1/3,0) rectangle (2+1/3,5);
\draw (0,0) rectangle (3,5);
\node [left] at (0,2.5) {$A_{i,j}$};
\draw [fill, xslant=1/10] (2-1/3,5) circle (1 pt);
\node[xslant=1/10, above] at (2,5) {$a$};
\draw[dotted] (1,0) -- (1,5);
\draw[red] (1,5) -- (2-1/2,5);
\node [red, above] at (1+1/3,5){$\leq E_{i} \alpha_{i,j}$};
\node at (1.8,1.2) {$C_i$};
\node at (1.6,3){$C_{i,1}$};
\node at (2.3,3){$C_{i,2}$};
\end{tikzpicture}
}
\caption{\label{figure - subcylinders} The two halves of the orbit-neighborhood $C_i = C_{i,1} \cup C_{i,2}$ which do not intersect any discontinuities of $f$ and whose sums differ by exactly $\sigma$ in the proof of Theorem \ref{theorem - essential points give finite essential values}.}
\end{figure}
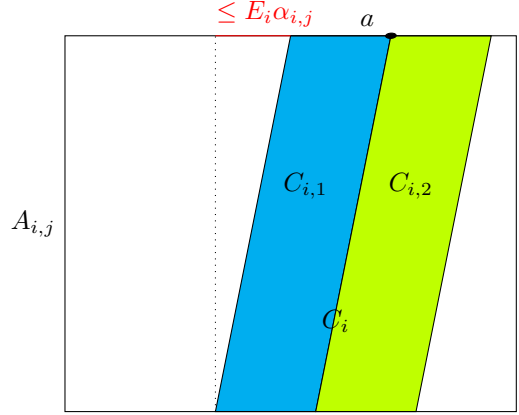

\begin{corollary}\label{corollary - generic generators}
If the closed subgroup of $G$ generated by the `jumps' $\sigma_{j,i}$ at essential points $a_{j,i}$ is all of $G$, then the geodesic flow on $\tilde{M}$ is ergodic in direction $\theta$ provided the underlying flow $\{M, \varphi_t\}$ is ergodic in direction $\theta$.  If all discontinuities $a_{i,j}$ are essential points, then every value $f_{i,j}$ is an essential value for $\{\tilde{M},\tilde{\varphi}_t\}$.
\end{corollary}

That is, for a generic method of making `cuts' $\gamma_i$ on the surface $M$ and using them to form a $G$-fold cover, as long as the values $f_{i,j}$ include generators for the group $G$, the extension is ergodic in almost every direction \textit{provided we may show the discontinuities are essential points in almost-every direction}.

\begin{example}
Let $M$ be the two-torus $\mathbb{T}=[0,1)^2$ with standard identifications (which is flat without any singularities).  Almost every direction $\theta$ has a good sequence of single-cylinder periodic directions given by a subsequence of the continued fraction convergents to $\theta$, and the flow $\{M, \varphi_t\}$ is ergodic in any irrational direction.  Let $\{g_i\}$ generate the group $G$, and for each $i$ construct a pair of parallel geodesics of equal length, at random starting points, in random directions, and of random lengths (with the relevant distributions are all absolutely continuous with respect to Lebesgue).  Define $f$ to be $\pm g_i$ on these pairs.  Then in any direction not parallel to such a segment, the projection of $f$ to the base segment $[0,1)\times\{0\}$ will be of bounded variation ($\var (f) = 4\sum|g_i|$) and have integral zero, so by a combination of Propositions \ref{prop - avoiding points generic in measure} and \ref{proposition - single cylinder avoiding points are essential}, Theorem \ref{theorem - essential points give finite essential values}, and Corollary \ref{corollary - generic generators}, for generic such constructions the skew product $\{\tilde{M}, \tilde{\varphi}_t\}$ is ergodic in almost every direction.

Instead of generating `parallel pairs' of cuts, we could generate \textit{single} cuts which are subdivided appropriately.  For example, suppose that that $a_ig_i + b_i g_i' =0$, where $a_i$ and $b_i$ are in $\mathbb{R}$, and the smallest closed subgroup in $G$ containing all of the $\{g_i,g_i'\}$ is $G$ itself.  Generate $i$ cuts of random length in random directions, and subdivide each cut into two portions whose ratio is $a_i/b_i$; let $f$ take value $g_i$ on the segment of proportional length $b_i$ and value $g_i'$ on the other.
\end{example}

Note that any `reasonable' method for generating cuts on $M$ which ensures that the transfer function $f$ has zero integral in almost-every direction will also generically ensure that the skew product $\{\tilde{M},\tilde{\varphi}_t\}$ is ergodic in almost every direction, provided that bounds on sums in cylinders may be established; by Proposition \ref{proposition - single cylinder avoiding points are essential} Veech surfaces which admit a single-cylinder direction will automatically obey such a bound.

\begin{example}
Let $M$ be \textit{square-tiled}: a finite cover of the torus such that $M$ may be considered to be tiled by squares which share corners.  Then $M$ is a Veech surface; this fact was first shown by Veech using Boshernitzan's criterion in \cite{veech}, but modern proofs center around the use of arithmetic subgroups of $\textrm{SL}_2(\mathbb{R})$ (see e.g. \cite[App. C]{hubert-lelievre}, \cite[Thm. 5.5]{gutkin-judge}).  Assume further that $M$ has a single-cylinder direction, so that almost every direction $\theta$ admits a good sequence of periodic single-cylinder directions.  By the same argument as in the previous example, then, for a generic method of constructing $f$ the corresponding flow is ergodic in almost every direction.
\end{example}

See Appendix \ref{section - make one cylinder directions} for material relating to determining which square-tiled surfaces have a single-cylinder direction as well as for a simple method of generating such surfaces.  While convenient, the assumption that $M$ admits a single-cylinder direction is not necessary to apply the ideas and results of this section; an example of a class of surfaces with no single-cylinder directions for which certain constructions generically produce ergodic covers is addressed in Example \ref{example - generic six tiles} and the remark following that example.


We remark that the idea of comparing ergodic sums on two different sides of a discontinuity goes back at least as far as \cite{schmidt} and can be found more recently in \cite{CG}, and the notion of finding generic points bounded a positive proportional distance from the orbits of singularities can be found in \cite{boshernitzan}.  A similar approach to studying ergodicity of real-valued skew products over rank-one interval exchange transformations can be found in \cite{chaika-hubert}.

\section{Generalized Staircases}

In this section we shall apply the same principles as in our generic constructions to certain \textit{specific} skew products of square-tiled surfaces called \textit{generalized staircases}.

\begin{definition}
Let $M$ be a translation surface, and let $\tilde{M}$ be a $\mathbb{Z}^d$ cover of $M$ formed by taking for each $i=1,2,\ldots,d$ a pair of parallel cuts of equal length, with $f = \pm e_i$ on each cut, where $e_i$ is the $i$-th standard basis element for $\mathbb{Z}^d$.  If all of the cuts intersect at most at their endpoints, then $\tilde{M}$ will be called a \textit{generalized staircase} over $M$; there is a natural visualization of the flow on the surface with some number of cuts which transfer the flow along a $\mathbb{Z}^d$-periodic cover.  If furthermore $M$ is square-tiled and each cut is of length one (the length of the sides of the squares) and all endpoints are integer points, $\tilde{M}$ will be called a \textit{natural staircase} over $M$.
\end{definition}

Note that if each cut is nontrivial in homology, then the $\mathbb{Z}^d$-cover of $M$ is \textit{unramified}, and if $M \in \mathcal{H}(2)$, by \cite[Thm. 1.4]{fraczek-ulcigrai} the system $\{\tilde{M},\tilde{\varphi}_t\}$ is \textit{not} ergodic in almost every direction.  However, if some of the cuts have endpoints which are \textit{not} in $D$ (the singularities of the flat matric on $M$), then we have constructed a ramified cover over a set of marked points $D$.  It is direct to see that if  $M$ consists of a single square then only trivial natural staircases are possible.  If we represent the torus with \textit{two} squares, however, then a natural staircase which is a ramified cover becomes possible.  This cover of the torus was presented in \cite{hooper-hubert-weiss}, where is was shown to be equivalent to a previously studied \textit{cylinder transformation}, known to be ergodic in all irrational directions.

Generalized staircases can be parameterized by the base points and lengths of their respective cuts.  The results of the previous section may be readily adapted under the addition  restriction that the cuts are parallel to sides to ensure that generic generalized staircases are ergodic in almost every direction.  However specific directions can have quite different behavior.

\begin{example}
Let $M = \mathbb{T}$ be the two-torus.  Suppose that $\theta$ is a direction which admits a good sequence of periodic approximations.  Then there exist uncountably many choices of $\beta, r$ such that with cuts as in Figure \ref{figure - conze nonregular example}, the skew product model of the corresponding generalized staircase $\{\tilde{\mathbb{T}},\tilde{\varphi}_t\}$ is \textit{nonregular} (the only essential values are $\{0,\infty\}$).  This staircase is the same as the skew product considered by Conze in \cite[Thm. 4.2]{conze-cocycles-rotation} where exactly this conclusion is shown.  Therefore the geodesic flow is \textit{recurrent} ($0$ is an essential value), but not ergodic (not all of $\mathbb{Z}$ are essential values).  
\end{example}


\begin{figure}
\centering{\begin{tikzpicture}
\draw (0,0) rectangle (4,4);
\draw[thick, red] (0,0) --  (1.2,0);
\node[below] at (.6,0){$+1$};
\node[above] at (1.2,0){$r$};
\draw[thick, red] (2,0) -- (3.2,0);
\node[below] at (2.6,0){$-1$};
\node[above] at (3.2,0){$\beta+r$};
\node[above] at (2,0){$\beta$};
\draw[->] (0,0) -- (.8,3.2);
\node[right] at (.4,1.6){$\theta$};
\draw[fill] (0,0) circle (1 pt);
\draw[fill] (1.2,0) circle (1 pt);
\draw[fill] (2,0) circle (1 pt);
\draw[fill] (3.2,0) circle (1 pt);
\end{tikzpicture}}
\caption{\label{figure - conze nonregular example}A generalized staircase over the $2$-torus which is nonregular for uncountably many parameters $\beta$, $r$.}
\end{figure}
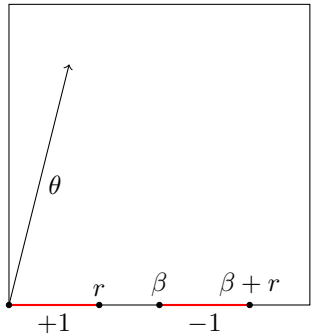

The situation of natural staircases is in some sense more interesting: the severe restrictions imposed by the integer endpoints of the cuts present an immediate obstacle to directly applying the techniques of the previous section without modification.  If the staircase is given by an unramified cover, then all endpoints of the cuts are singularities, and there is no adapting our techniques (indeed, for $M \in \mathcal{H}(2)$ the flow will not be ergodic in almost every direction by \cite{fraczek-ulcigrai}).  So assume that at least one of the cuts has an endpoint (with integer coordinates, as we consider a natural staircase) which is not a singularity.  In a square-tiled surface, the directions which admit representation by periodic cylinders correspond to rational-slope flows on the plane.  The endpoints of our cuts are all integer points in the plane, as are all of the singularities $\tilde{D}$; the endpoints of our cuts will always belong to the edge of the periodic cylinders and therefore orbit into singularities in the directions of the periodic approximations, a rather insurmountable impediment to directly apply Definition \ref{definition - self avoiding}.

Thus far, however, the only identifications of the edges of the cylinders we have used is the identification of the top to the bottom, but there are also certain identifications on the \textit{sides} of the cylinders.  It is illustrative to make explicit which directions for square-tiled surfaces form different types of periodic approximation.  Let $M$ be square-tiled on $k$ squares.  The flow in direction $\theta$ induces a rotation by $\theta' = 1/\theta \mod 1$ on $\onetorus$ when the bases of all squares are projected to a single circle.  Let the continued fraction expansion of $\theta'$ be given by
\[ \theta' = [a_1,a_2,a_3,\ldots].\] Then the periodic approximations corresponding to the sequence of the convergents $p_{n}/q_{n}$ to $\theta'$.  The relative error $E_n$ of the approximation is no more than $k q_n\|q_{n} \theta'\|$ (the standard error of $\|q_{n} \theta'\|$ multiplied by the number of squares, as a single cylinder must wind through all of them, and divided by the width $1/q_n$), and it is standard that
\[ \frac{k}{a_{n+1}+1} \leq k q_n\|q_{n} \theta\| \leq \frac{k}{a_{n+1}}.\]
Therefore a good sequence of periodic approximations corresponds to a subsequence of partial quotients diverging, while a fairly good sequence 
of periodic approximations corresponds to a subsequence of partial quotients at least as large as $2k+1$, but bounded above.  Both situations are 
clearly satisfied for generic $\theta'$.  It is not immediately clear that if $M$ admits a single-cylinder direction that both conditions are generically 
satisfied along the subsequence of $p_n/q_n$ which correspond to these single-cylinder directions, but this fact holds by the Patterson/Sullivan theorem 
used in Lemma \ref{lemma - approximate full measure of directions}.

\begin{theorem}\label{theorem - get essential values in staircases}
Suppose that $M$ is square-tiled and $f$ defines a natural staircase.  Assume that $d$ is not a singularity of the flat metric, but is an endpoint of a
cut along which $f$ takes the value $e_{i}$ (and not the end point of any other cut).
Assume further that along a subsequence of fairly good periodic approximations there exist cylinders $A_{i,j}$  with $d$ on \textit{both} sides of $A_{i,j}$, and for each of these approximations $\theta_i$ there is some $P$ (independent of $i$) such that for each $x \in A_{i,j}$
\[ \| S_{h_{i,j}}(x)\| \leq P,\] where the sum is with respect to the flow in the periodic direction $\theta_i$.  Then $e_i$ is a finite essential value of $\{\tilde{M},\tilde{\varphi}_t\}$ in direction $\theta$. If $d$ is the endpoint of several cuts,
then the essential value is  given by the sum of the values of $f$ on the cuts terminating at $d$.
In both cases the staircase is recurrent.
\end{theorem}
\begin{proof}
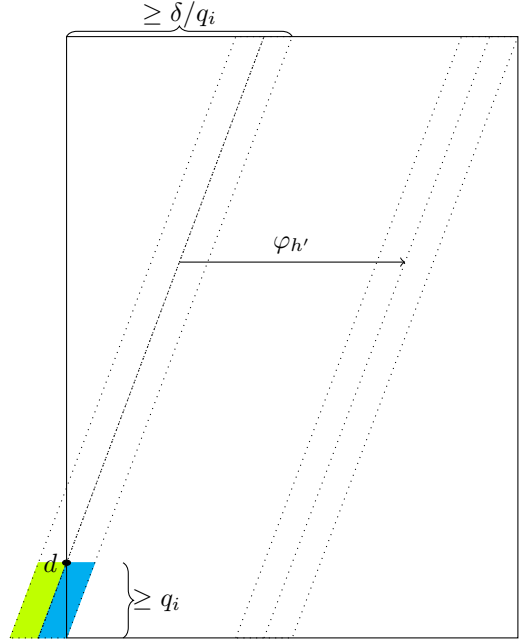
\begin{figure}[ht]
\centering{\begin{tikzpicture}[xscale=1.5]
\filldraw[draw=lime, fill=lime] (-.5,0) -- (-.25,1) -- (0,1) -- (-.25,0) -- (-.5,0);
\filldraw[draw=cyan, fill=cyan] (-.25,0) -- (0,1) -- (.25,1) -- (0,0) -- (-.25,0);
\draw[dotted] (-.5,0) -- (1.5,8) -- (1.75,8) -- (-.25,0)--(-.5,0);
\draw[dotted] (-.25,0) -- (0,0) -- (2,8) -- (1.75,8) -- (-.25,0);
\draw (0,0) rectangle (4,8);
\draw[fill] (0,1) circle (1 pt);
\node[left] at (0,1){$d$};
\draw[dotted] (-.25,0) -- (1.75,8);
\draw[dotted] (1.5,0) -- (3.5,8) -- (4,8) -- (2,0) -- (1.5,0);
\draw[dotted] (1.75,0) -- (3.75,8);
\draw[->] (1,5) -- node[above]{$\varphi_{h'}$} (3,5);
\draw[snake = brace, segment amplitude = 4pt](0,8) -- node[above, anchor=south]{$\geq \delta/q_i$} (2,8);
\draw[snake =brace, segment amplitude=4pt](.5,1)-- node[right]{$\geq q_i$} (.5,0);
\end{tikzpicture}}
\caption{\label{figure - neighborhoods for staircases} A modified orbit-neighborhood used in staircases: $C_{i,1}$ and $C_{i,2}$ are the two colored parallelograms.  The region outside the cylinder appears somewhere inside via the appropriate identification and there are no singularities within $C_{i,1}$ or $C_{i,2}$.  No cuts have endpoints within these neighborhoods other than the specified cuts with values $\pm e_i$ with common endpoint $d$.  The time $h'$ flow (to traverse the entire vertical distance) acts as horizontal translation by $E_n/q_n \in [\delta/q_n,1/2q_n]$ on this set. The constant $\delta$ is from the definition of fairly good approximation.}
\end{figure}

The singularities and ramification points are arranged along the edges of the periodic cylinders $A_{i,j}$.  The flow in the direction $\theta_i$ induces 
the rotation by $p_i/q_i$ on the projection of all horizontal bases of the $k$ squares to a single circle.  Note that the vertical distance between successive
singularities and ramification points on the boundary of the cylinder (i.e.\ the geodesic distance in direction $\theta_i$ between them) must be at 
least $q_i$, and the height of any cylinder is no more than $kq_i$, where $k$ is the number of squares in $M$.  Consider the flow of a vertical 
interval of length $q_i$ both above and below $d$ on both sides of $A_{i,j}$, and arrange the cylinder so that this interval is at the bottom of one side.  We flow this interval forward until intersecting the height of $d$ within the cylinder, and backwards to the height of the base of the cylinder; this notion is well-defined despite the flow exiting the cylinder in one of these directions as there are no singularities near $d$ (see Figure \ref{figure - neighborhoods for staircases}).  No other
discontinuities of $f$ are in these vertical segments, then, and furthermore the cuts with endpoint $d$ cuts across the cylinder $A_{i,j}$ from exactly one
of the two copies of $d$ on the boundary of $A_{i,j}$ (because $d$ is an endpoint of the cut).  See Figure \ref{figure - neighborhoods for staircases}; the two halves of this orbit neighborhood are $C_{i,1}$ and $C_{i,2}$; because there are no singularities within height $q_i$ of $d$ along the boundary of the cylinder there is no problem extending these neighborhoods outside the boundary.  The distances are maintained as drawn, and any cut besides those with endpoint $d$ which might intersect these neighborhoods must pass completely through them; the only endpoint of any cut seen in $C_{i,1}$ and $C_{i,2}$ is $d$.  The difference in the ergodic sums is therefore exactly a sum of $\pm e_i$, the values of $f$ on the cuts terminating at $d$.  As the error $E_i$ is not larger than half the width of $A_{i,j}$ (as we have a fairly good sequence of periodic directions), these neighborhoods may flow for time to traverse the entire height of the cylinder without intersecting the sides (except for the small bit at the bottom where there are no singularities, as mentioned prior).  The measure of each orbit neighborhood can be measured by the product of the vertical length of the segment between $d$ and the nearest singularity (which is at least $q_i$ as we are square-tiled) and the horizontal drift; the error term $E_i$, which is a positive portion of $1/q_i$ as we have assumed only a fairly good sequence of periodic approximations; the measure of each $C_{i,m}$ is at least $\delta$, where $\delta>0$ is the fixed constant from the definition of a fairly good sequence of periodic approximations.

Therefore the sets $C_{i,1}$ and $C_{i,2}$ again form a quasi-rigidity sequence of sets (their near-invariance under $\varphi_1$ is clear from construction).  As in Theorem \ref{theorem - essential points give finite essential values}, the
the sum of the values of $f$ on the cuts terminating at $d$ is the difference in sums between infinitely many of the two $C_{i,1}$ and $C_{i,2}$, and is therefore an essential value.
\end{proof}

\begin{corollary}\label{corollary - ergodic natural staircases}
Suppose that for each $i$, at least one endpoint one of the cuts where $f$ takes value $e_i$ is not a singularity and is not the endpoint of any other cut, and assume that $M$ admits a single-cylinder direction.  Then the natural staircase given by $f$ is ergodic in almost every direction.
\begin{proof}
As $M$ admits a single-cylinder periodic direction, almost every $\theta$ has a fairly good sequence of single-cylinder periodic directions by Lemma \ref{lemma - approximate full measure of directions}.  As there is only one cylinder to consider, the relevant bound on sums used in Theorem \ref{theorem - get essential values in staircases} is satisfied by Proposition \ref{proposition - single cylinder avoiding points are essential}.  As we assume that each cut has a non-singularity endpoint not shared by any other cut, the group of essential values therefore contains each $e_i$.
\end{proof}
\end{corollary}

\begin{example}
Consider the surface on at least five squares given by Figure \ref{figure - the 5 square surface with cuts labelled}; one row of two squares and one row of at least three.  The first presentation with cuts as marked on the left and its associated staircase is \textit{not} ergodic in almost every direction by \cite[Thm. 1.4]{fraczek-ulcigrai}; The surface is in $\mathcal{H}(2)$ and both cuts are nontrivial in $H_1(M,\mathbb{Z})$. The presentation on the right, however, has both cuts with an endpoint which is not a singularity; by Theorem \ref{theorem - get essential values in staircases}, then, the flow on the corresponding staircase is ergodic in almost every direction.  Note that the two compact surfaces are the same, and the only difference is in the choice of cuts!
\end{example}

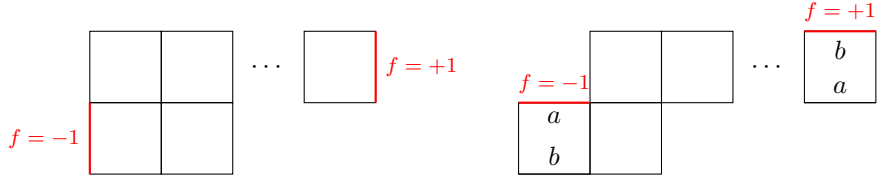
\begin{figure}[htb]
\centering{\begin{tikzpicture}[scale=.95]
\draw (0,0) rectangle (1,1);
\draw (1,0) rectangle (2,1);
\draw (0,1) rectangle (1,2);
\draw (1,1) rectangle (2,2);
\node at (2.5,1.5){$\cdots$};
\draw (3,1) rectangle (4,2);
\draw[thick, red] (4,1) -- node[right]{\footnotesize{$f=+1$}}(4,2);
\draw[thick, red] (0,0) -- node[left]{\footnotesize{$f=-1$}}(0,1);
\draw (6,0) rectangle (7,1);
\draw (7,0) rectangle (8,1);
\draw (8,1) rectangle (9,2);
\draw (7,1) rectangle (8,2);
\node at (9.5,1.5){$\cdots$};
\draw (10,1) rectangle (11,2);
\draw (10,1) -- node[above]{$a$} (11,1);
\draw (6,1) -- node[below]{$a$} (7,1);
\draw (10,2) -- node[below]{$b$} (11,2);
\draw (6,0) -- node[above]{$b$} (7,0);
\draw[thick, red] (10,2) -- node[above]{\footnotesize{$f=+1$}} (11,2);
\draw[thick,red](6,1) -- node[above]{\footnotesize{$f=-1$}} (7,1);
\end{tikzpicture}}
\caption{\label{figure - the 5 square surface with cuts labelled}Both surfaces are equivalent (given the same number of omitted squares), but the two different realizations as a natural staircase are quite different.  Identification is by opposite-side except where marked with $a,b$.}
\end{figure}

The condition that there is some direction $\theta'$ such that $\{M, \varphi_t\}$ decomposes into a single periodic cylinder is convenient but not necessary:

\begin{example}\label{example - generic six tiles}
Let $M$ be the square-tiled surface in Figure \ref{figure - six squares}, and assume that the number of tiles in each row is odd and at least five.  Then if we construct a $\mathbb{Z}$-staircase using the cuts marked, this staircase is ergodic in almost every direction.  This surface, however, admits no single-cylinder directions.
\end{example}

\begin{proof}
That each row has more than three squares ensures that the two cuts with value $+1$ do not join together to form a loop; each cut of value $+1$ therefore has both endpoints not singularities.  The assumption that the number of squares in each row is odd guarantees the existence of an integer involution point; a vertex around which we may rotate the entire figure by $\pi$.  Suppose that some flow in slope $q/p>1$ is a single-cylinder direction.  This flow induces a rotation by $p/q$ on the bases of the squares; a periodic cylinder will be of width $1/q$ and may be taken to begin with an interval of length $1/q$ extending horizontally from this involution point.  That this surface has no single-cylinder directions follows from Corollary \ref{corollary - no single cylinder}, where this particular example is referenced.

The surface clearly has a decomposition into two cylinders of equal measure and length: the horizontal flow achieves such a decomposition, and the two cylinders are exchanged by the involution.  By setting the involution point to be the origin, the involution itself may be represented as $-\textrm{Id}$.  Of course $-\textrm{Id}$ is in the center of $\Gamma$, the Veech group.  It follows that any image of these two cylinders under any element of $\Gamma$ also consists of two components which are exchanged by the involution.

Note that $f$, the function used to define the staircase, is invariant under this involution.  Combining this fact with the fact that the involution exchanges the two cylinders, we see that the ergodic sum in each periodic cylinder is equal to the other.  As their sum must be the average of the function, each periodic cylinder sees a total sum of zero.  We may then consider the endpoint of the cut with value $+1$ and directly apply the same construction in Figure \ref{figure - neighborhoods for staircases} to find an essential value of $\pm 1$ for this staircase.

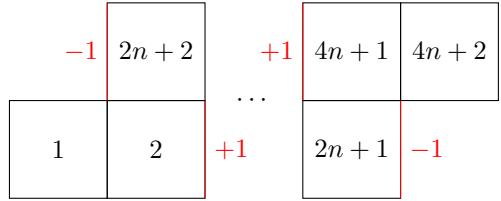
\begin{figure}[htb]
\center{\begin{tikzpicture}[scale=1.3]
\draw (0,0) rectangle node{$1$} (1,1);
\draw (1,0) rectangle node{$2$} (2,1);
\draw (1,1) rectangle node{$2n+2$} (2,2);
\draw (3,0) rectangle node{$2n+1$} (4,1);
\draw (3,1) rectangle node{$4n+1$} (4,2);
\draw (4,1) rectangle node{$4n+2$} (5,2);
\node at (2.5,1){$\cdots$};
\draw[red] (1,1) -- node[left]{$-1$} (1,2);
\draw[red] (2,0) -- node[right]{$+1$} (2,1);
\draw[red] (3,1) -- node[left]{$+1$}(3,2);
\draw[red] (4,0) -- node[right]{$-1$}(4,1);
\end{tikzpicture}}
\caption{\label{figure - six squares}The surface $M$ for example \ref{example - generic six tiles}; the labels one through $4n+2$ on the squares are used in Appendix \ref{section - make one cylinder directions}.  Opposite sides are identified; the surface is in $\mathcal{H}(1,1)$ and has no single-cylinder directions.  We form a staircase using the indicated values on the specified unit intervals.}
\end{figure}

\end{proof}

Note that any probabilistic technique of generating covers (not restricted to natural staircases, so the cuts are simple geodesic segments as in \S\ref{section - generic}) on these surfaces which respects the following
\begin{itemize}
\item The flow in the horizontal direction of \textit{every} point sees a sum of zero across all cuts traversed through one period,
\item the cuts are generated pairwise, invariant under the involution,
\item the values taken on the cuts generate the group $G$
\end{itemize}
will generically produce covers which are ergodic in almost-every direction, exactly via the techniques of \S\ref{section - generic}.


\appendix

\section{Permutations and parallelogram-tiled surfaces}\label{section - make one cylinder directions}
While the aforementioned results of Hubert-Lelievre, McMullen, Kontsevich-Zorich and Lanneau-Nguyen guarantee for certain square-tiled $M$ the existence of directions $\theta_i$ which decompose $M$ into a single periodic cylinder (providing the bound on ergodic sums necessary in the definition of essential points, Definition \ref{definition - essential point}), it is of interest to quickly determine which square-tiled surfaces admit such a direction in general.

Any square-tiled translation surface, with $k$ total squares, may be identified with a pair of permutations in $\Sigma_k$ (the permutation group of $k$ elements) $(\sigma_h, \sigma_v)$.  From a combinatorial perspective it is unimportant that the tiles are squares; any parallelogram tiling would suffice (and still allows a notion of `top/bottom' and `left/right').  The permutation $\sigma_h$ is given by the identification of the right side of the tile indexed by $i$ with the left side of the tile indexed by $\sigma_h(i)$, while the top of tile $i$ is identified with the bottom of the tile of index $\sigma_v(i)$: see Figure \ref{figure - permutation for tiling}.  The observation that square-tiled surfaces may be indexed by pairs of permutations appears in \cite{schmithuesen}.  

\begin{figure}[htb]
\center{
\begin{tikzpicture}[scale =2]
\draw (0,0) rectangle node{$i$} (1,1);
\draw (0,1) rectangle node{$\sigma_v(i)$} (1,2);
\draw (1,0) rectangle node{$\sigma_h(i)$} (2,1);
\end{tikzpicture}
}
\caption{\label{figure - permutation for tiling} The permutations $\sigma_h$ and $\sigma_v$ used to identify the square-tiled $M$.  $\sigma_h$ tracks identification of the horizontal flow, and $\sigma_v$ tracks the vertical flow.}
\end{figure}
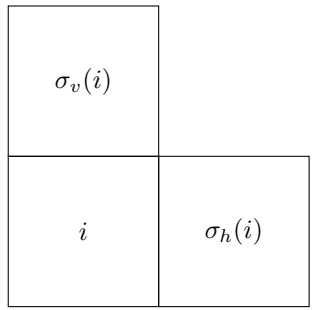

It is clear that any choice of $(\sigma_h, \sigma_v)$ along with a choice of parallelogram tile defines a square-tiled translation surface $M$, and that any choice of charts for a surface $M$ which gives a tiling by parallelograms can be identified with such a pair of permutations.  In the search for single-cylinder directions the distinction between squares and parallelograms is unimportant (the two being distinguished only by an affine transformation), so we will use squares without loss of generality.

\begin{theorem}\label{theorem - stratum information}
The parallelogram-tiled translation surface identified by $(\sigma_h,\sigma_v)$ is connected if and only if $\langle \sigma_h,\sigma_v \rangle$, the subgroup of $\Sigma_k$ generated by $\sigma_h$ and $\sigma_v$, acts transitively on $\mathbb{Z}_k$.  The cycle decomposition of the commutator 
\[\tilde{\sigma}=[\sigma_v,\sigma_h]=\sigma_v^{-1} \circ \sigma_h^{-1} \circ \sigma_v \circ \sigma_h\]
encodes the stratum of $M$: the number of distinct vertices is given by the number of cycles in $\tilde{\sigma}$, and each corresponding vertex has cone angle $2 \pi$ times the length of the cycle.
\begin{proof}
Suppose that there is a proper subset $S \subset \mathbb{Z}_k$ which is invariant under each of $\sigma_h$, $\sigma_v$.  Then the collection of squares indexed by $S$ contains no identification to squares indexed by $\mathbb{Z}_k \setminus S$, thereby forming at least two connected components.  On the other hand, if there are no proper invariant subsets, then given any two indices $i$ and $j$, there is a sequence of $\sigma_{t_n}$ which connect the square indexed by $i$ to that indexed by $j$ through some sequence of side identifications (where $t_n \in \{h,v\}$).

Consider the upper-right corner of the tile indexed by $i$.  The commutator $\tilde{\sigma}$ exactly encodes the trajectory of a counter-clockwise circle around this corner, and the conclusion is direct from this observation.  Let $V$ be the number of cycles in $[\sigma_v,\sigma_h]$.  Then $M$ is given as a surface with $k$ faces (the tiles), $2k$ edges (each edge is used in exactly two tiles) and $V$ vertices.  By computing the Euler characteristic, the genus $g$ is given by
\[g = \frac{k-V}{2}+1,\]
assuming that $M$ is connected (otherwise we could consider each connected component separately).
\end{proof}
\end{theorem}

\begin{corollary}
Let $\tilde{\sigma}$ be an arbitrary element of the commutator subgroup of $\Sigma_k$.  Then the number of cycles in $\tilde{\sigma}$ is either odd or even according to whether $k$ is odd or even.
\begin{proof}
It is not difficult to show (for example, by induction on $n$) that the commutator subgroup of $\Sigma_k$ is $A_k$, the alternating subgroup, and that $A_k$ is in fact exactly the set of commutators.  The claim may be then shown directly for any element of $A_n$.  However, it may also be derived here from the Euler characteristic and Theorem \ref{theorem - stratum information}: the genus must be an integer.  In the event that $\tilde{\sigma}$ is the commutator of two permutations with a shared proper invariant subset, the corresponding translation surface is not connected, but we may apply the same reasoning to each connected component.
\end{proof}
\end{corollary}

Now suppose that we flow along $M$ in a rational direction (recall that without loss of generality we are using unit squares as tiles).  Assume that the slope is $q/p \geq 1$, so that in the time necessary for the flow to traverse a unit of vertical distance, it traverses a horizontal distance of $p/q$.  Then we define a new tiling of $M$ by $kq$ parallelograms: the new tile of index $(i-1)q+j$ (for $i=1,\ldots,k$ and $j=1,\ldots,q$) is defined to be the parallelogram given by flowing the interval $[(j-1)/q,j/q)$ in the base of the square of index $i$ for a vertical distance of one.  See Figure \ref{figure - retiling}.  If the flow were of slope less than one, we could perform an analogous retiling using intervals of the vertical edges of the original tiles, flowed to a horizontal distance of one.

The new tiling can also be described with a pair of permutations $\sigma_h'$, $\sigma_v'$.  We are free to choose which direction of these new parallelograms should correspond to `vertical'; whether the direction $q/p$ was larger or less than one, there is a natural choice.  As our figure uses a flow of slope larger than one, we let the flow correspond to the `vertical' direction in the new tiles.

\begin{figure}[htb]
\center{
\begin{tikzpicture}
\draw (0,0) rectangle (5,5);
\filldraw[fill=cyan, draw=black] (1,0) -- (3,5) -- (4,5) -- (2,0) -- (1,0);
\node[below] at (1,0){$\frac{j-1}{q}$};
\node[below] at (2,0){$\frac{j}{q}$};
\draw[->] (.5,.5) -- node[left]{$q/p$} (8/5+1/2,4.5);
\draw[->] (2.5,2.5) --  (3.5,2.5) node[below]{$\sigma'_h$};
\draw[->] (3.4,4.5) --  (3.8,5.5) node[right]{$\sigma'_v$};
\end{tikzpicture}
}
\caption{\label{figure - retiling} Retiling the surface $M$ based on the flow in direction $q/p\geq 1$.  If the large square is the original tile of index $i$, then the highlighted parallelogram is the new tile of index $(i-1)q+j$.}
\end{figure}
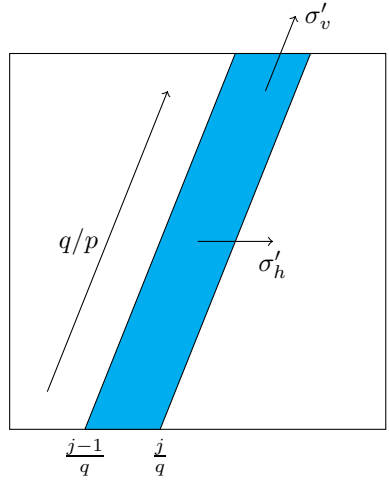

\begin{lemma}
The new permutations are defined by
\[ \sigma_h'(iq+j) = \begin{cases} (i-1)q+j+1 & (j<q)\\ (\sigma_h(i)-1)q+1 & (j=q) ,\end{cases}\]
\[ \sigma_v'(iq+j) = \begin{cases} (\sigma_v(i)-1)q+j+p & (j < q-p) \\ ((\sigma_v \circ \sigma_h)(i)-1)q+j+p-q& (j\geq q-p).\end{cases}\]
\begin{proof}
The horizontal identification is direct (refer again to Figure \ref{figure - retiling}).  For the vertical identifications, note that flowing to traverse one unit of vertical distance, we either flow only through the top of a square, or first across the right edge and then through the top, exactly according to whether the base of the interval is less than $(1-p/q)$ or not.
\end{proof}
\end{lemma}

These permutations are of limited practical use: the number of tiles increases with the complexity of the slope.  We may, however, reduce the problem to a permutation on $k$ elements if we relax the requirements on the parallelograms.  In a tiling, it is necessary for all singularities of the metric to belong to the set of corners of the tiles.  Without this requirement, however, we may cover $M$ with $k$ parallelograms with sides horizontal and of slope $q/p$:  the parallelogram of index $i$ is given by flowing the interval $[0,1/q)$ on the bottom of the square of index $i$ through a vertical distance of $q$.

It is direct to see that the $k$ parallelograms defined in this way are disjoint: the induced rotation on the bases is of $p/q$ for each unit of vertical distance traversed, so a vertical distance of $q$ is traversed before the left endpoint of this interval intersects an integer point.  These tiles must therefore cover $M$, as we construct $k$ such tiles and each has equal measure to the original squares.  It is not a proper tiling, however, as the singularities will presumably belong to the sides of these parallelograms. Now compare with the proper retiling by $kq$ parallelograms, each of the $k$ parallelograms in the improper tiling is tiled by $q$ proper parallelograms, so the proper tiling is a subtiling of the improper tiling. 
The permutation $\hat{\sigma} \in \Sigma_k$ 
\[\hat{\sigma}(i) = (\sigma_v')^q ((i-1)q+1)\]
is the permutation induced by this subtiling in the flow direction.

Then the following is clear by our construction:
\begin{theorem}
The flow in slope $q/p$ is a single-cylinder direction if and only if $\hat{\sigma}$ is a single cycle on $k$ elements.
\end{theorem}

Generation of $\hat{\sigma}$ from the original pair $(\sigma_h,\sigma_v)$ is intimately related to Farey series, the Stern-Brocot tree, and a notion from combinatorics (used mostly in computer science) called \textit{Lyndon words}.  On $\{0,1\}^{\mathbb{N}}$, define the homomorphism $\phi$ by
\[ \phi(0)= \sigma_v, \quad \phi(1) = \sigma_v \circ \sigma_h,\] and we construct words $\omega_{p/q}$ of length $q$ with exactly $p$ ones in the following manner, beginning with $\omega_0=0$ and $\omega_1=1$: let $p/q<p'/q'$ and assume that $p'q-pq'=1$.  Let $p''/q''=(p+p')/(q+q')$ be the \textit{Farey sum} (the `bad arithmetic sum').  Then $\omega_{p''/q''}=\omega_{p/q} \omega_{p'/q'}$.

See \cite{lothaire} for background in combinatorics on words; we will not present that background here.  A binary word of finite length is called a \textit{Lyndon word} if each left factor has an average strictly smaller than the overall average of the word, i.e. it is strictly smaller in lexicographic order than all of its rotations.

Then our Farey-series-based construction generates all (binary, finite) Sturmian Lyndon words; the correspondence is made quite explicit by considering the geometric construction of Sturmian words by counting crossings of vertical versus horizontal integer-grid segments, and the Lyndon property is quickly seen to be inherited by the given order of composition under the assumption that both $\omega_{p/q}$ and $\omega_{p'/q'}$ are Lyndon words.  That of all primitive Sturmian words of a given length and density there is only one which is also a Lyndon word is fairly direct to show; a proof appears in \cite{ralston2}.

Note that we defined $\hat{\sigma}$ to track the orbit of the interval $[0,1/q)$ on the base of a square.  Considering a different interval or a vertical segment would amount to performing a fixed cyclic transformation on the sequence of permutations used to define $\hat{\sigma}$.  That is, instead of considering
\[\hat{\sigma} = \tau_1 \circ \tau_2 \circ \cdots \circ \tau_n\]
where each $\tau_i$ is either $\sigma_h$ or $\sigma_v$, we would consider 
\[\tau_k \circ \tau_{k+1} \circ \cdots \circ \tau_n \circ \tau_1 \circ \cdots \circ \tau_{k-1}.\]  This new permutation is naturally seen to be in the same conjugacy class as $\hat{\sigma}$, and therefore has the same number and length of cycles.  Similarly, as the labeling of the tiles was arbitrary, it is only the conjugacy class/cycle decomposition of the resulting permutation that is of interest.  It would be of considerable interest to classify the permutation pairs $(\sigma_h, \sigma_v)$ on $k$ elements which generate a ``Farey-Sturmian" composition in the described manner which is a $k$-cycle.  Together with the cycle decomposition of the commutator $[\sigma_v, \sigma_h]$, knowing which pairs can generate $k$-cycles would classify the square-tiled surfaces with a single-cylinder direction!

\begin{corollary}\label{corollary - no single cylinder}
If $\sigma_h$ and $\sigma_v$ generate a subgroup of $\Sigma_k$ which contains no $k$-cycle, then $M$ admits no single-cylinder direction.
\begin{proof}
For slopes larger than one the result is immediate.  For slopes less than one, we would perform an analogous construction as indicated previously with the same result.  Alternately, any cylinder decomposition for a slope less than one may be mapped via the Veech group to a slope larger than one.
\end{proof}
\end{corollary}

Consider for example the surfaces given in Figure \ref{figure - six squares}, with squares labeled as in that figure:
\[ \sigma_h = (1 \ldots (2n+1))((2n+2) \ldots (4n+2)),\] \[\sigma_v = (1)(2 (2n+2))(3 (2n+3)) \ldots ((2n+1)(4n+1))(4n+2).\]
Both $\sigma_h$ and $\sigma_v$ are in the alternating group $A_6$ ($\sigma_h$ is the product of two disjoint cycles of equal length, and $\sigma_v$ is explicitly given as an even number of transpositions), but a $4n+2$-cycle is a composition of $4n+1$ transpositions.  Therefore by the above corollary, \textit{this surface admits no single-cylinder directions}.

\begin{example}The \textit{eierlegende Wollmilchsau} is the surface on eight squares defined by
\[ \sigma_h=(1234)(5678), \quad \sigma_v=(1836)(2745).\]  This surface is of genus $3$ with four singularities of cone angle $4 \pi$: from the permutation presentation, the number of vertices (and therefore the genus) and their cone angles may be read off the commutator of $\sigma_v$ and $\sigma_h$: 
\[ [\sigma_v, \sigma_h]=(13)(24)(57)(68).\]  Both $\sigma_h$ and $\sigma_v$ are in $A_8$, while a complete $8$-cycle has seven transpositions.  Therefore this surface has no single-cylinder directions.  This result also follows from the fact that the Veech group is all of $\textrm{SL}_2(\mathbb{Z})$ \cite{herr-schmith}.  Therefore all periodic decompositions are of the same type: two cylinders of equal heights and measures.
\end{example}

\begin{corollary}
Let $\sigma_v$ and $\sigma_h$ define a (connected) square-tiled surface, and let $\tilde{\sigma}=[\sigma_v,\sigma_h]$.  If $\tilde{\sigma}$ contains at least $d$ cycles of length one, then it is possible to construct a $\mathbb{Z}^d$ natural staircase whose geodesic flow is ergodic in almost every direction.  On the other hand, if there are no cycles of length one and $M \in \mathcal{H}(2)$, then no natural staircase is ergodic in almost every direction.
\begin{proof}
The statement about the possibility of ergodic $d$-dimensional natural staircases follows from Corollary \ref{corollary - ergodic natural staircases}, using the vertices which are not singularities (fixed points under $\tilde{\sigma}$) as endpoints of distinct cuts whose values generate $\mathbb{Z}^d$. If there are no cycles of length one, then all integer points are singularities.  All cuts are therefore nontrivial in the relative homology group $H_1(M, D, \mathbb{Z})$.  In the event that $M \in \mathcal{H}(2)$, then, the geodesic flow is not ergodic in almost every direction, again citing \cite{fraczek-ulcigrai}.
\end{proof}
\end{corollary}

The above techniques may be used to quickly generate surfaces with single-cylinder directions.  Assuming that a surface is tiled by parallelograms one through $k$ such that vertical flow is a single cylinder, the permutation $\sigma_h$ therefore encodes the stratum of the surface.  Letting the tiles be given in order, so that $\sigma_v(i)=i+1 \mod k$, the singularities may be deduced from the cycle decomposition of the commutator
\[[\sigma_v,\sigma_h](i)= \sigma_h^{-1}\left(\sigma_h(i)+1\right)-1.\]

For arbitrary dimension $d$ the techniques of this appendix may be used to study $d$-cube tiled translation manifolds through the use of $d$ permutations (for the relevant face identifications in each dimension).  A single-cylinder direction therefore gives rise to a parallelepiped tiling for which one of the permutations is a single cycle.  The geometric data is still encoded by the permutations, if more involved to extract.  For example, the number of connected components still corresponds directly to the number of invariant sets in the group generated by the given permutations, and the singularities of the flat metric at certain $(d-2)$-dimensional simplices involves commutators of the permutations.  As geodesic flows in higher dimensions are poorly understood at present, we have phrased all the above in the particular case $d=2$.

\bibliography{ergodic-extensions-bifile}
\bibliographystyle{plain}

\end{document}